\documentclass[12pt]{article}
\usepackage[top=1in,bottom=1in,left=1in,right=1in]{geometry}
\usepackage{indentfirst}
\usepackage{amsfonts,amsmath,amsthm,amssymb}
\usepackage{cite}
\usepackage[hidelinks]{hyperref}

\newtheorem{theorem}{Theorem}[section]
\newtheorem{lemma}{Lemma}[section]

\theoremstyle{remark}
\newtheorem{remark}{Remark}[section]

\newcommand{\Z}{\mathbb Z}

\newcommand{\indd}{\operatorname{ind}}

\newcommand{\Um}{U_m}

\begin{document}
\title{A Multiplicative Fourier Proof of the Length-Four Index Conjecture}
\author{
Hongjian Li$^{1,}$\footnote{E-mail\,$:$ lhj@gdufs.edu.cn. Supported by the Project of Guangdong University of Foreign Studies (Grant No. 2024RC063).}\quad
Pingzhi Yuan$^{2,}$\footnote{E-mail\,$:$ yuanpz@scnu.edu.cn. Supported by the National Natural Science Foundation of China (Grant No. 12571003) and the Basic and Applied Basic Research Foundation of Guangdong Province (Grant No. 2024A1515010589).}\quad
Shijie Yuan$^{2,}$\footnote{ E-mail\,$:$ yuanshijie202505@163.com.}\quad 
Weilin Zhang$^{3,}$\footnote{Corresponding author. E-mail\,$:$ weilin@gzhu.edu.cn.}\\
{\small\it $^{1}$School of Mathematics and Statistics, Guangdong University of Foreign Studies,}\\
{\small\it Guangzhou 510006, Guangdong, P. R. China}\\
{\small\it $^{2}$School of Mathematical Sciences, South China Normal University,}\\
{\small\it Guangzhou 510631, Guangdong, P. R. China}\\
{\small\it $^{3}$School of Mathematics and Information Science, Guangzhou University,}\\
{\small\it Guangzhou 510006, Guangdong, P. R. China}
}

\date{}
 \maketitle

\noindent{\bf Abstract}\quad 
Let $C_n$ be a cyclic group of order $n$. We prove that if $(n,6)=1$, then every minimal zero-sum sequence of length four over $C_n$ has index one, thereby resolving the length-four index conjecture. After the gcd reduction, the nonunit case follows from the theorem of Shen--Xia--Li, and the remaining unit case is solved by a new multiplicative Fourier argument. The index-two residue identity yields a character-moment relation, and the odd characters with vanishing first moment form an exceptional spectrum of size at most $157\varphi(n)/1440<\varphi(n)/9$. A finite-group uncertainty
principle then forces the four-term multiset to be invariant under negation, contradicting minimality. Apart from standard facts about primitive Dirichlet $L$-functions, the remaining argument is finite and requires neither asymptotic estimates nor computational verification.

\medskip



\section{Introduction}

Zero-sum theory constitutes a central subfield of additive combinatorics and algebraic number theory, with close connections to Davenport-type invariants, non-unique factorization theory in Krull
monoids, Dedekind sums, discrepancy theory, and Heegaard Floer homology~\cite{CFS99,Gero87,Gero09,GaGe09,Ge16,Ge21,JRW13}. For broader background on structural additive theory and zero-sum problems over finite abelian groups, we refer to \cite{Gry12}. The study of zero-sum phenomena over finite abelian groups has led to a range of structural and extremal questions, among which the index of short zero-sum sequences over cyclic groups plays a prominent role.

Let $C_n\cong \Z/n\Z$ be written additively. A sequence
\[
S=(a_1)\cdots(a_k)
\]
over $C_n$ is called a \emph{zero-sum sequence} if $a_1+\cdots+a_k=0$ in $C_n$, and it is \emph{minimal zero-sum} if no nonempty proper subsequence is zero-sum. If $(x)_n\in\{0,1,\dots,n-1\}$ denotes the least nonnegative residue of $x$ modulo $n$, the index of $S$ is
\[
\indd(S)=\min_{g\in(\Z/n\Z)^\times}\frac{1}{n}\sum_{j=1}^k(ga_j)_n.
\]
For a zero-sum sequence with nonzero terms, the quantity inside the minimum is a positive integer.

The length-four index conjecture, formulated by Ponomarenko~\cite{Po04}, asserts that
\begin{equation}\label{eq:index-conj}
(n,6)=1,\qquad |S|=4,\qquad S\text{ minimal zero-sum}
\quad\Longrightarrow\quad
\indd(S)=1.
\end{equation}

The index problem belongs to the broader additive theory of zero-sum and modular-sum phenomena in cyclic groups~\cite{LK89,Gao00}. One may more generally ask for which pairs $(k,n)$ every length-$k$ minimal zero-sum sequence over
$\mathbb Z/n\mathbb Z$ has index $1$; such pairs are termed \emph{good pairs} \cite{Ge21}. The classification of good pairs is known outside the length-four case. For $k\leq3$, the corresponding pairs are good \cite{Po04}. For $5\leq k\leq\lfloor n/2\rfloor+1$, every pair $(k,n)$ is known to be bad~\cite{Ge21}. For $k>\lfloor n/2\rfloor+1$, independent works of Savchev--Chen and Yuan show that every pair is good~\cite{SC07,Yuan07}; see also~\cite{Gao00}. Thus the case $k=4$ is the remaining central case in this classification and leads precisely to \eqref{eq:index-conj}.

Over the subsequent two decades, the conjecture was attacked by algebraic, combinatorial, computational, and analytic methods, yielding a hierarchy of partial results. An important early breakthrough was the prime-power case: Li, Plyley, Yuan, and Zeng proved in 2010 that the conjecture holds whenever $n=p^\alpha$ for a prime $p$ with $(p,6)=1$ \cite{LPYZ10}. Their argument reduces the problem to structural control of the residues of sequence elements modulo prime powers, using a detailed analysis of parameter ranges and fractional ceiling estimates to locate a unit multiplier $m\in(\mathbb Z/n\mathbb Z)^\times$ realizing index $1$. A series of subsequent works established the conjecture under
several additional structural hypotheses~\cite{Xia13,LP13,SX14}. The full case in which the group order has two distinct prime divisors was proved by Xia and Shen~\cite{XS13}; see also~\cite{SXL14} for a shorter proof and related refinements.

A further advance was obtained by Zeng and Qi~\cite{ZQ17}. Their contribution treats the case in which all terms generate the ambient cyclic group, while the complementary case had already been settled in earlier work. Together these results imply the conjecture under the stronger hypothesis $(n,30)=1$. Consequently, among moduli satisfying $(n,6)=1$, this settles the case $5\nmid n$ and isolates $5\mid n$ as the principal remaining difficulty, a point also emphasized in later work~\cite{GV20,Ge21}.

Further structural progress was obtained by Ge, who proved the conjecture for the class of singular sequences~\cite{Ge18}. Most earlier approaches were combinatorial or arithmetic. A different method was introduced by Ge in 2021, who brought Fourier analysis and discrepancy estimates into the study of the index conjecture~\cite{Ge21}. Ge's geometric reformulation turns the index condition into a lattice-point covering problem: one seeks a unit $g\in(\mathbb Z/n\mathbb Z)^\times$ for which three scaled residues $(ga_1)_n/n,(ga_2)_n/n,(ga_3)_n/n$ lie in the half-unit interval $I=(0,1/2)$. A smoothed periodic indicator with compact Fourier support is used to approximate the characteristic function of $I$, and the resulting Ramanujan-sum cross terms are controlled by estimates for gcd growth of linear forms modulo composite $n$. This method proved the conjecture for all $n>10^{20}$, thereby reducing the problem to a finite range of moduli.

Pendleton~\cite{Pe25} subsequently sharpened Ge's Fourier-analytic framework and reduced the explicit asymptotic threshold from $10^{20}$ to $4.6\times10^{13}$. The improvement comes from refined estimates for the Ramanujan-sum contributions and an optimized choice of the smoothing parameter in the Fourier approximation. Pendleton also computationally verified the conjecture for all $n<1.8\times10^6$ using high-performance computation and
strengthened structural lemmas governing forbidden linear relations of the form $x\pm3y\equiv0$ and $3x\pm y\equiv0\pmod n$.

We remove the size restriction completely and prove the following theorem.

\begin{theorem}\label{thm:full}
Let $n$ be a positive integer with $(n,6)=1$. Then every minimal zero-sum sequence of length four over $C_n$ has index one.
\end{theorem}

Our proof combines the previously known nonunit case with a new treatment of the unit case. After the standard gcd reduction, a minimal counterexample with at least one nonunit term is excluded by the theorem of Shen--Xia--Li~\cite[Theorem~1.3]{SXL14}. It therefore remains to consider the case in which all four terms are units modulo
$n$. In this case, the index-two condition yields the exact identity
\[
\sum_{j=1}^{4}(ga_j)_n=2n
\qquad (g\in U_n).
\]
Applying multiplicative characters to this identity leads to a relation governed by the first moments of Dirichlet characters.
An explicit factorization of these moments isolates an exceptional set of odd characters, whose cardinality is less than
$\varphi(n)/9$. A finite-group uncertainty principle then forces the four-term multiset to be invariant under negation, contradicting minimality. Apart from standard facts about primitive Dirichlet $L$-functions, the remaining argument is finite and requires neither smoothing nor asymptotic error terms.

The paper is organized as follows. Section~2 reviews the necessary character-theoretic preliminaries, derives the multiplicative moment identity, and establishes the required bound for the exceptional odd spectrum. Section~3 performs the minimal-modulus reduction, separates the nonunit and unit cases, and completes the proof of
Theorem~\ref{thm:full} by means of the finite-group uncertainty principle.


\section{Preliminaries}\label{sec:preliminaries}

We recall the standard character-theoretic notation used throughout
the proof. 

Throughout this section, let $n>1$. Set
\[
U_n=(\mathbb Z/n\mathbb Z)^\times,
\qquad
|U_n|=\varphi(n),
\]
and let $\widehat{U_n}$ denote the character group of $U_n$. We
identify the elements of $U_n$ with their representatives in
$\{1,2,\ldots,n-1\}$.

The principal character is denoted by $\chi_0$. Since
$(-1)^2=1$ in $U_n$, every $\chi\in\widehat{U_n}$ satisfies
\[
\chi(-1)\in\{1,-1\}.
\]
We call $\chi$ \emph{even} if $\chi(-1)=1$ and \emph{odd} if
$\chi(-1)=-1$. Since character values on $U_n$ are roots of unity,
\[
\chi(x)^{-1}=\overline{\chi(x)}
\qquad (x\in U_n).
\]

We shall also use the standard orthogonality relation
\[
\sum_{x\in U_n}\chi(x)
=
\begin{cases}
\varphi(n), & \chi=\chi_0,\\
0, & \chi\ne\chi_0.
\end{cases}
\]

As usual, a character $\chi\in\widehat{U_n}$ is extended to a
Dirichlet character modulo $n$ by setting
\[
\chi(a)=0
\qquad\text{whenever }(a,n)>1.
\]
We use the standard notions of induced and primitive Dirichlet
characters. Thus, if $\chi$ is induced by a primitive character
$\psi$ modulo $f\mid n$, then $f$ is the conductor of $\chi$.

For $\chi\in\widehat{U_n}$, define its first moment by
\begin{equation}\label{eq:first-moment}
M_n(\chi)
:=
\frac1n\sum_{x\in U_n}x\chi(x).
\end{equation}


We first derive the basic moment identity underlying the multiplicative Fourier argument.

\begin{lemma}
\label{lem:moment-identity}
Let $S=(a_1)(a_2)(a_3)(a_4)$ be a zero-sum sequence over $C_n$ with
$a_1,a_2,a_3,a_4\in U_n$ and $\operatorname{ind}(S)=2$. Then, for every nonprincipal character
$\chi\in\widehat{U_n}$,
\begin{equation}\label{eq:moment-identity}
M_n(\chi)
\sum_{j=1}^{4}\overline{\chi(a_j)}
=0.
\end{equation}
\end{lemma}

\begin{proof}
For $g\in U_n$, set
\[
t(g):=\frac1n\sum_{j=1}^{4}(ga_j)_n.
\]
Since $S$ is zero-sum, $t(g)$ is an integer. As
$a_j\in U_n$, each $(ga_j)_n$ lies in $\{1,\ldots,n-1\}$, and hence
\[
t(g)\in\{1,2,3\}.
\]
Moreover,
\[
(-ga_j)_n=n-(ga_j)_n
\qquad (1\le j\le4),
\]
so
\[
t(g)+t(-g)=4.
\]
Because $\operatorname{ind}(S)=2$, we have $t(g)\ge2$ for every
$g\in U_n$. Applying the same inequality to $-g$ gives
$t(-g)\ge2$, and therefore
\[
t(g)=t(-g)=2
\qquad (g\in U_n).
\]
Equivalently,
\begin{equation}\label{eq:index-two-residue}
\sum_{j=1}^{4}(ga_j)_n=2n
\qquad (g\in U_n).
\end{equation}

Now let $\chi\in\widehat{U_n}$ be nonprincipal. Multiplying
\eqref{eq:index-two-residue} by $\chi(g)$ and summing over
$g\in U_n$, the right-hand side vanishes by character orthogonality:
\[
\sum_{g\in U_n}\chi(g)=0.
\]
Thus
\[
\sum_{j=1}^{4}
\sum_{g\in U_n}(ga_j)_n\chi(g)=0.
\]

Fix $j$. Since $a_j\in U_n$, multiplication by $a_j$ is a bijection
of $U_n$. With the change of variables $x=ga_j$, we have
$g=xa_j^{-1}$, and hence
\[
\chi(g)
=
\chi(x)\chi(a_j^{-1})
=
\chi(x)\overline{\chi(a_j)}.
\]
Therefore
\[
\begin{aligned}
\sum_{g\in U_n}(ga_j)_n\chi(g)
&=
\overline{\chi(a_j)}
\sum_{x\in U_n}x\chi(x)\\
&=
n\,\overline{\chi(a_j)}\,M_n(\chi).
\end{aligned}
\]
Summing over $j=1,\ldots,4$ gives
\[
n\,M_n(\chi)
\sum_{j=1}^{4}\overline{\chi(a_j)}=0.
\]
Dividing by $n$ proves
\[
M_n(\chi)
\sum_{j=1}^{4}\overline{\chi(a_j)}=0.
\]
\end{proof}


Lemma~\ref{lem:moment-identity} shows that, for every nonprincipal
character $\chi$ with $M_n(\chi)\ne0$,
\[
\sum_{j=1}^{4}\overline{\chi(a_j)}=0.
\]
Thus the only characters for which the moment identity may fail to
force this vanishing are those satisfying $M_n(\chi)=0$. We now
determine precisely when this occurs.

\begin{lemma}
\label{lem:moment-factorization}
Let $\chi$ be a nonprincipal Dirichlet character modulo $n$, induced
by a primitive character $\psi$ of conductor $f\mid n$. Then
\begin{equation}\label{eq:moment-factorization}
M_n(\chi)
=
M_f(\psi)
\prod_{\substack{p\mid n\\ p\nmid f}}
\bigl(1-\psi(p)\bigr).
\end{equation}
\end{lemma}

\begin{proof}
Let
\[
\mathcal P
=
\{p:p\mid n,\ p\nmid f\},
\qquad
Q:=\prod_{p\in\mathcal P}p,
\]
with $Q=1$ if $\mathcal P=\varnothing$.

Since $\chi$ is induced by $\psi$, we have
\[
\chi(a)=\psi(a)
\qquad\text{whenever }(a,n)=1.
\]
Hence
\[
\sum_{x\in U_n}x\chi(x)
=
\sum_{\substack{1\le a\le n\\(a,n)=1}}a\psi(a).
\]
Moreover, $\psi(a)=0$ whenever $(a,f)>1$. Thus it remains only
to exclude those additional prime divisors in $\mathcal P$. Therefore, by inclusion--exclusion,
\begin{align*}
\sum_{x\in U_n}x\chi(x)
& = \sum_{\substack{1\le a\le n\\(a,Q)=1}}a\psi(a) \\
& = \sum_{1\le a\le n}a\psi(a)\sum_{D\mid (a, Q)}\mu_{\mathrm{Mob}}(D)\\
& = \sum_{D\mid Q}\mu_{\mathrm{Mob}}(D)
\sum_{\substack{1\le a\le n\\D\mid a}}
a\psi(a),
\end{align*}
where $\mu_{\mathrm{Mob}}$ denotes the M\"obius function. Writing
$a=Db$ and using $(D,f)=1$, we obtain
\[
\sum_{x\in U_n}x\chi(x)
=
\sum_{D\mid Q}
\mu_{\mathrm{Mob}}(D)D\psi(D)
\sum_{1\le b\le n/D}b\psi(b).
\]

Since $Df\mid n$, put
\[
q_D:=\frac{n}{Df}.
\]
Using the periodicity of $\psi$ modulo $f$, we have
\[
\begin{aligned}
\sum_{1\le b\le n/D}b\psi(b)
&=
\sum_{r=0}^{q_D-1}
\sum_{c=1}^{f}(rf+c)\psi(c)\\
&=
f\left(\sum_{r=0}^{q_D-1}r\right)
\left(\sum_{c=1}^{f}\psi(c)\right)
+
q_D\sum_{c=1}^{f}c\psi(c).
\end{aligned}
\]
Since $\chi$ is nonprincipal, its primitive inducing character
$\psi$ is also nonprincipal. Hence
\[
\sum_{c=1}^{f}\psi(c)=0.
\]
Consequently,
\[
\sum_{1\le b\le n/D}b\psi(b)
=
\frac{n}{Df}
\sum_{c=1}^{f}c\psi(c).
\]
Substituting this identity above gives
\[
\sum_{x\in U_n}x\chi(x)
=
\frac nf
\left(\sum_{c=1}^{f}c\psi(c)\right)
\sum_{D\mid Q}\mu_{\mathrm{Mob}}(D)\psi(D).
\]
Finally,
\[
\sum_{D\mid Q}\mu_{\mathrm{Mob}}(D)\psi(D)
= \sum_{A\subseteq\mathcal P}
(-1)^{|A|}
\prod_{p\in A}\psi(p) = 
\prod_{p\in\mathcal P}\bigl(1-\psi(p)\bigr).
\]
Since $\psi(c)=0$ whenever $(c,f)>1$, we have
\[
\frac1f\sum_{c=1}^{f}c\psi(c)
=
M_f(\psi).
\]
Therefore, dividing by $n$ gives
\[
M_n(\chi)
=
M_f(\psi)
\prod_{\substack{p\mid n\\p\nmid f}}
(1-\psi(p)),
\]
as required.
\end{proof}


\begin{lemma}
\label{lem:primitive-nonvanishing}
If $\psi$ is a primitive odd Dirichlet character modulo $f$, then
\[
M_f(\psi)\ne0.
\]
\end{lemma}

\begin{proof}
Since $\psi$ is odd, it is nonprincipal. By
\cite[Theorem~12.20, p.~268]{Apostol1976},
\[
L(0,\psi)
=
-\frac1f\sum_{a=1}^{f}a\psi(a)
=
-M_f(\psi).
\]
It therefore suffices to show that $L(0,\psi)\ne0$.

Let
\[
G(1,\psi)
:=
\sum_{a=1}^{f}\psi(a)e^{2\pi i a/f}
\]
be the Gauss sum associated with $\psi$. For a primitive character
$\psi$, the functional equation
\cite[\S12.11, Eq.~(16), p.~263]{Apostol1976} gives
\[
L(1-s,\psi)
=
\frac{f^{\,s-1}\Gamma(s)}{(2\pi)^s}
\left(
e^{-\pi i s/2}
+\psi(-1)e^{\pi i s/2}
\right)
G(1,\psi)L(s,\overline{\psi}).
\]
Since $\psi$ is odd, $\psi(-1)=-1$. Setting $s=1$ therefore yields
\[
L(0,\psi)
=
-\frac{i}{\pi}\,
G(1,\psi)L(1,\overline{\psi}).
\]
Because $\psi$ is primitive,
\[
|G(1,\psi)|^2=f,
\]
by \cite[Theorem~8.15(c)]{Apostol1976}; in particular,
$G(1,\psi)\ne0$. Moreover, $\overline{\psi}$ is nonprincipal, and the classical nonvanishing theorem gives
\[
L(1,\overline{\psi})\ne0;
\]
see \cite[p.~149]{Apostol1976}.
Consequently, $L(0,\psi)\ne0$, and hence $M_f(\psi)\ne0$.
\end{proof}


Define the exceptional odd spectrum by
\begin{equation}\label{eq:Zn-def}
Z_n
:=
\left\{
\chi\in\widehat{U_n}:
\chi(-1)=-1,\;
M_n(\chi)=0
\right\}.
\end{equation}
The set $Z_n$ consists precisely of the odd characters for which
Lemma~\ref{lem:moment-identity} gives no information on the character
sum, since the coefficient $M_n(\chi)$ vanishes. Our next goal is to
describe these characters explicitly and then estimate their number.

Let $\chi$ be an odd character modulo $n$, induced by the primitive
character $\psi$ of conductor $f\mid n$. Since induction preserves
parity,
\[
\psi(-1)=\chi(-1)=-1,
\]
so $\psi$ is also odd. Lemma~\ref{lem:primitive-nonvanishing} therefore
gives
\[
M_f(\psi)\ne0.
\]
Combining this with the factorization
\eqref{eq:moment-factorization}, we obtain
\begin{equation}\label{eq:moment-zero-characterization}
M_n(\chi)=0
\quad\Longleftrightarrow\quad
\psi(p)=1
\ \text{for some }p\mid n\text{ with }p\nmid f.
\end{equation}
Thus the vanishing of $M_n(\chi)$ is caused entirely by primes dividing
$n$ that are absent from the conductor of $\chi$.

We now translate this condition into the Chinese remainder
decomposition of the character group. From this point to the end of
the section, assume $(n,6)=1$, and write
\[
n=\prod_{i=1}^{r}q_i,
\qquad
q_i=p_i^{\alpha_i},
\qquad
5\le p_1<\cdots<p_r.
\]
By the Chinese remainder theorem,
\[
U_n\cong\prod_{i=1}^{r}U_{q_i},
\qquad
\widehat{U_n}\cong\prod_{i=1}^{r}\widehat{U_{q_i}}.
\]
Accordingly, every character $\chi\in\widehat{U_n}$ can be written
uniquely in the form
\[
\chi=\prod_{i=1}^{r}\chi_i,
\qquad
\chi_i\in\widehat{U_{q_i}}.
\]

For each $i$, put
\[
m_i:=\frac{n}{q_i},
\]
and let $\chi^{(i)}\in\widehat{U_{m_i}}$ denote the character formed
from all components $\chi_j$ with $j\ne i$. The prime $p_i$ is absent
from the conductor of $\chi$ precisely when the local component
$\chi_i$ is principal. In this case, the primitive inducing character
satisfies
\[
\psi(p_i)=\chi^{(i)}(p_i).
\]
Moreover, since $\chi_i$ is principal,
\[
\chi(-1)=\chi^{(i)}(-1).
\]
Consequently, for such a prime $p_i$, the two conditions
\[
\chi(-1)=-1,
\qquad
\psi(p_i)=1
\]
are equivalent to
\[
\chi^{(i)}(-1)=-1,
\qquad
\chi^{(i)}(p_i)=1.
\]

This motivates the definition
\[
E_i
:=
\left\{
\chi\in\widehat{U_n}:
\chi(-1)=-1,\;
\chi_i=1,\;
\chi^{(i)}(p_i)=1
\right\},
\]
where $1$ denotes the principal character of $U_{q_i}$.
By \eqref{eq:moment-zero-characterization}, every character in $Z_n$
belongs to at least one of the sets $E_i$, and conversely every
character in some $E_i$ belongs to $Z_n$. Hence
\begin{equation}\label{eq:Zn-union}
Z_n=\bigcup_{i=1}^{r}E_i.
\end{equation}
The union need not be disjoint, since a character may satisfy the
vanishing condition for more than one prime divisor of $n$.

Finally, because $\chi_i$ is principal for $\chi\in E_i$, the map
\[
\chi\longmapsto\chi^{(i)}
\]
identifies $E_i$ with the set
\[
\left\{
\rho\in\widehat{U_{m_i}}:
\rho(p_i)=1,\;
\rho(-1)=-1
\right\}.
\]
If
\[
H_i:=\langle p_i\bmod m_i\rangle\le U_{m_i},
\]
then the condition $\rho(p_i)=1$ is equivalent to $\rho$ being
trivial on $H_i$. Thus estimating $|E_i|$ reduces to counting
characters of a finite abelian group that are trivial on a prescribed
subgroup but take the value $-1$ at the element $-1$. The following
elementary lemma gives exactly this count.

\begin{lemma}
\label{lem:character-counting}
Let $A$ be a finite abelian group, let $H\le A$, and let $z\in A$
have order two. Among the characters $\rho\in\widehat A$ that are
trivial on $H$, the number satisfying $\rho(z)=-1$ is
\[
\begin{cases}
0, & z\in H,\\[1mm]
\dfrac{|A|}{2|H|}, & z\notin H.
\end{cases}
\]
\end{lemma}

\begin{proof}
Let
\[
H^\perp
:=
\{\rho\in\widehat A:\rho(h)=1\text{ for every }h\in H\}
\]
be the annihilator of $H$. The characters in $H^\perp$ are precisely
the characters of $A$ that are trivial on $H$.

Every $\rho\in H^\perp$ induces a character $\widetilde\rho$ of the
quotient group $A/H$ by
\[
\widetilde\rho(aH):=\rho(a).
\]
This is well defined: if $aH=bH$, then $a^{-1}b\in H$, and hence
\[
\rho(b)
=
\rho(a)\rho(a^{-1}b)
=
\rho(a).
\]
Conversely, every character of $A/H$ lifts to a character of $A$
that is trivial on $H$. Thus
\[
H^\perp\cong\widehat{A/H}.
\]
Since a finite abelian group and its character group have the same
order,
\[
|H^\perp|
=
|\widehat{A/H}|
=
|A/H|
=
\frac{|A|}{|H|}.
\]

If $z\in H$, then every $\rho\in H^\perp$ is trivial on $z$, so
\[
\rho(z)=1.
\]
Hence no such character satisfies $\rho(z)=-1$.

Now suppose that $z\notin H$. Since $z$ has order two in $A$, we have
\[
(zH)^2=H.
\]
Moreover, $zH\ne H$ because $z\notin H$. Therefore $zH$ is a
nontrivial element of order two in $A/H$.

Consider the evaluation map
\[
\operatorname{ev}_{zH}:
\widehat{A/H}\longrightarrow\{\pm1\},
\qquad
\eta\longmapsto\eta(zH).
\]
This is a group homomorphism. Indeed,
\[
(\eta_1\eta_2)(zH)
=
\eta_1(zH)\eta_2(zH).
\]
Furthermore, since $zH$ has order two,
\[
\eta(zH)^2
=
\eta((zH)^2)
=
\eta(H)
=
1,
\]
so $\eta(zH)\in\{\pm1\}$ for every $\eta\in\widehat{A/H}$.

Because $zH$ is nontrivial, characters of the finite abelian group
$A/H$ separate points. Hence there exists
$\eta_0\in\widehat{A/H}$ such that
\[
\eta_0(zH)\ne1.
\]
Since the only possible values are $\pm1$, it follows that
\[
\eta_0(zH)=-1.
\]
Thus $\operatorname{ev}_{zH}$ is surjective.

Consequently its kernel has index two in $\widehat{A/H}$, and hence
\[
|\ker(\operatorname{ev}_{zH})|
=
\frac12|\widehat{A/H}|
=
\frac{|A|}{2|H|}.
\]
The fibre over $-1$ is a coset of the kernel, so it has the same
cardinality. Therefore exactly
\[
\frac{|A|}{2|H|}
\]
characters in $H^\perp$ satisfy $\rho(z)=-1$.
\end{proof}


If $r=1$, then $n=p^\alpha$ for some prime $p$. An odd character
cannot have conductor $1$, so its conductor is divisible by $p$.
Hence there is no prime $q\mid n$ with $q\nmid f$, and
\eqref{eq:moment-zero-characterization} gives
\[
Z_n=\varnothing.
\]
 We may therefore assume $r\ge2$. 
 
 For $1\le i\le r$, put
\[
d_i:=\operatorname{ord}_{m_i}(p_i).
\]
Since $r\ge2$ and $(n,6)=1$, each $m_i>1$ is odd, so $-1$
has order two in $U_{m_i}$. Applying Lemma~\ref{lem:character-counting} with
\[
A=U_{m_i},
\qquad
H=\langle p_i\bmod m_i\rangle,
\qquad
z=-1,
\]
gives
\[
|E_i|
\le
\frac{\varphi(m_i)}{2d_i}.
\]
Consequently, by \eqref{eq:Zn-union},
\begin{equation}\label{eq:Zn-union-bound}
|Z_n|
\le
\frac12\sum_{i=1}^{r}\frac{\varphi(m_i)}{d_i}
=
\frac{\varphi(n)}2
\sum_{i=1}^{r}
\frac{1}{\varphi(q_i)d_i}.
\end{equation}

It therefore remains to obtain a uniform lower bound for the
multiplicative orders $d_i$.

\begin{lemma}
\label{lem:order-lower-bound}
For every $i<r$,
\[
d_i\ge r-i+2.
\]
\end{lemma}

\begin{proof}
Fix $i<r$, and put
\[
P_i:=\prod_{j=i+1}^{r}p_j.
\]
Since $P_i\mid m_i$ and
\[
p_i^{d_i}\equiv1\pmod{m_i},
\]
we have
\[
P_i\mid p_i^{d_i}-1.
\]
Every prime divisor of $P_i$ is larger than $p_i$, and hence
\[
(P_i,p_i-1)=1.
\]
Therefore
\[
P_i
\mid
\frac{p_i^{d_i}-1}{p_i-1}
=
1+p_i+\cdots+p_i^{d_i-1}.
\]

Suppose that $d_i\le r-i+1$. Then
\[
P_i
\le
\frac{p_i^{d_i}-1}{p_i-1}
\le
\frac{p_i^{r-i+1}-1}{p_i-1}
<
\frac{p_i^{r-i+1}}{p_i-1}.
\]
On the other hand, since the $r-i$ primes
$p_{i+1},\ldots,p_r$ are strictly larger than $p_i$,
\[
P_i\ge(p_i+2)p_i^{r-i-1}.
\]
Since $p_i>2$,
\[
(p_i+2)(p_i-1)>p_i^2,
\]
and hence
\[
(p_i+2)p_i^{r-i-1}
>
\frac{p_i^{r-i+1}}{p_i-1},
\]
a contradiction. Hence $ d_i\ge r-i+2$.
\end{proof}


\begin{lemma}
\label{lem:exceptional-spectrum}
If $(n,6)=1$, then
\[
|Z_n|
\le
\frac{157}{1440}\varphi(n)
<
\frac{\varphi(n)}9.
\]
\end{lemma}

\begin{proof}
The case $r=1$ has already been disposed of.

Suppose first that $r=2$. If
\[
(p_1,p_2)=(5,7),
\]
then
\[
\operatorname{ord}_{7}(5)=6,
\qquad
\operatorname{ord}_{5}(7)=4.
\]
Passing to prime powers cannot decrease multiplicative order, and
hence
\[
\sum_{i=1}^{2}\frac{1}{\varphi(q_i)d_i}
\le
\frac1{4\cdot6}
+
\frac1{6\cdot4}
=
\frac1{12}
<
\frac{157}{720}.
\]
For every other pair, $p_2\ge11$. By
Lemma~\ref{lem:order-lower-bound}, $d_1\ge3$, while $d_2\ge1$.
Thus
\[
\sum_{i=1}^{2}\frac{1}{\varphi(q_i)d_i}
\le
\frac1{3(p_1-1)}
+
\frac1{p_2-1}
\le
\frac1{12}+\frac1{10}
=
\frac{11}{60}
<
\frac{157}{720}.
\]

For $r=3$, Lemma~\ref{lem:order-lower-bound} gives
$d_1\ge4$ and $d_2\ge3$. Hence
\[
\sum_{i=1}^{3}\frac{1}{\varphi(q_i)d_i}
\le
\frac1{4\cdot4}
+
\frac1{6\cdot3}
+
\frac1{10}
=
\frac{157}{720}.
\]

For $r=4$,
\[
\sum_{i=1}^{4}\frac{1}{\varphi(q_i)d_i}
\le
\frac1{4\cdot5}
+
\frac1{6\cdot4}
+
\frac1{10\cdot3}
+
\frac1{12}
=
\frac5{24}
<
\frac{157}{720}.
\]

For $r=5$,
\[
\sum_{i=1}^{5}\frac{1}{\varphi(q_i)d_i}
\le
\frac1{4\cdot6}
+
\frac1{6\cdot5}
+
\frac1{10\cdot4}
+
\frac1{12\cdot3}
+
\frac1{16}
=
\frac{137}{720}
<
\frac{157}{720}.
\]

It remains to consider $r\ge6$. Since the $p_i$ are distinct odd primes with $p_1\ge5$, we obtain
\[
p_i\ge2i+3,
\qquad
\varphi(q_i)\ge p_i-1\ge2(i+1).
\]
Hence Lemma~\ref{lem:order-lower-bound}, together with $d_r\ge1$, yields
\begin{equation}\label{eq:Dr-def}
\sum_{i=1}^{r}\frac{1}{\varphi(q_i)d_i}
\le
\sum_{i=1}^{r-1}
\frac{1}{2(i+1)(r-i+2)}
+
\frac{1}{2(r+1)}
=:D_r.
\end{equation}
A partial-fraction summation gives
\begin{equation}\label{eq:Dr-formula}
D_r
=
\frac{H_r+H_{r+1}-\frac52}{2(r+3)}
+
\frac{1}{2(r+1)},
\end{equation}
where
\[
H_r:=1+\frac12+\cdots+\frac1r.
\]
Moreover,
\begin{equation}\label{eq:Dr-difference}
D_{r+1}-D_r
=
-
\frac{
4(r+1)(r+2)H_r-r(7r+17)
}{
4(r+1)(r+2)(r+3)(r+4)
}.
\end{equation}
For $r\ge3$,
\[
H_r\ge H_3=\frac{11}{6},
\]
and hence
\[
4(r+1)(r+2)H_r-r(7r+17)
\ge
\frac{r^2+15r+44}{3}
>0.
\]
Thus $D_r$ is strictly decreasing for $r\ge3$. In particular, for
$r\ge6$,
\[
D_r\le D_6=\frac{67}{315}<\frac{157}{720}.
\]

Consequently, in every case,
\[
\sum_{i=1}^{r}
\frac{1}{\varphi(q_i)d_i}
\le
\frac{157}{720}.
\]
Substitution into \eqref{eq:Zn-union-bound} gives
\[
|Z_n|
\le
\frac{157}{1440}\varphi(n)
<
\frac{\varphi(n)}9.
\]
\end{proof}


\section{Proof of the main theorem}

There is no length-four minimal zero-sum sequence over $C_1$.
Hence any counterexample has $n>1$.

For a function $h:G\to\mathbb C$ on a finite abelian group $G$,
write
\[
\operatorname{supp}h
:=
\{x\in G:h(x)\ne0\}
\]
for its support. For $\gamma\in\widehat G$, define the unnormalized
Fourier transform by
\[
\widehat h(\gamma)
:=
\sum_{x\in G}h(x)\overline{\gamma(x)}.
\]

We shall use the classical uncertainty principle for Fourier analysis
on finite abelian groups; see
\cite[Theorem~4.10]{Na20}. The uncertainty principle states that if $h\ne0$, then
\begin{equation}\label{eq:uncertainty}
|\operatorname{supp}h|\,
|\operatorname{supp}\widehat h|
\ge |G|.
\end{equation}

The inequality is independent of the normalization convention used
for the Fourier transform.

\begin{proof}[Proof of Theorem~\ref{thm:full}]
Suppose, to the contrary, that the theorem is false. Choose $n$
minimal, subject to $(n,6)=1$, for which there exists a length-four
minimal zero-sum sequence
\[
S=(a_1)(a_2)(a_3)(a_4)
\]
over $C_n$ with
\[
\operatorname{ind}(S)\ne1.
\]
Represent the four nonzero terms by integers
\[
1\le a_j\le n-1.
\]

We first show that
\begin{equation}\label{eq:primitive-gcd}
\gcd(n,a_1,a_2,a_3,a_4)=1.
\end{equation}
Suppose otherwise, and put
\[
d:=\gcd(n,a_1,a_2,a_3,a_4)>1,
\qquad
m:=\frac nd,
\qquad
b_j:=\frac{a_j}{d}.
\]
Since $1\le a_j\le n-1$ and $a_j=db_j$, we have
\[
1\le b_j\le m-1.
\]
Moreover, since
\[
\sum_{j=1}^{4}a_j\equiv0\pmod n,
\]
division by $d$ gives
\[
\sum_{j=1}^{4}b_j\equiv0\pmod m.
\]
Hence
\[
S'=(b_1)(b_2)(b_3)(b_4)
\]
is a zero-sum sequence over $C_m$. It is minimal, since any proper
zero-sum subsequence of $S'$ would, after multiplication by $d$,
give a proper zero-sum subsequence of $S$.

We claim that
\[
\operatorname{ind}(S')\ne1.
\]
Suppose instead that $\operatorname{ind}(S')=1$. Then there exists
$h\in\Um$ such that
\[
\sum_{j=1}^{4}(hb_j)_m=m.
\]

We next lift $h$ to a unit modulo $n$. Let $\mathcal Q$ be the set of
prime divisors of $n$ that do not divide $m$. By the Chinese remainder
theorem, there exists an integer $u$ satisfying
\[
u\equiv h\pmod m,
\qquad
u\equiv1\pmod q
\quad(q\in\mathcal Q).
\]
Every prime divisor of $n$ either divides $m$ or belongs to
$\mathcal Q$, and hence $(u,n)=1$. Thus $u\in U_n$.

Since $n=dm$, $a_j=db_j$, and $u\equiv h\pmod m$,
\[
(ua_j)_n
=
d(hb_j)_m.
\]
Therefore
\[
\sum_{j=1}^{4}(ua_j)_n
=
d\sum_{j=1}^{4}(hb_j)_m
=
dm=n,
\]
which gives
\[
\operatorname{ind}(S)=1,
\]
a contradiction.

Hence $S'$ is again a counterexample. But
\[
m<n
\qquad\text{and}\qquad
(m,6)=1,
\]
contradicting the minimality of $n$. This proves
\eqref{eq:primitive-gcd}.

We now distinguish two cases.

\medskip
\noindent
\emph{Case 1. At least one $a_j$ is a nonunit modulo $n$.}

By \eqref{eq:primitive-gcd},
\[
\gcd(n,a_1,a_2,a_3,a_4)=1.
\]
Hence the theorem of Shen--Xia--Li
\cite[Theorem~1.3]{SXL14} gives
\[
\operatorname{ind}(S)=1,
\]
contrary to the choice of $S$.

\medskip
\noindent
\emph{Case 2. All $a_1,a_2,a_3,a_4$ are units modulo $n$.}

Thus
\[
a_1,a_2,a_3,a_4\in U_n.
\]
Since $S$ is a counterexample, its index is not $1$. On the other
hand, for every $g\in U_n$,
\[
\frac1n\sum_{j=1}^{4}(ga_j)_n
+
\frac1n\sum_{j=1}^{4}(-ga_j)_n
=
4.
\]
Hence the index of a length-four zero-sum sequence with nonzero terms
is at most $2$. Therefore
\[
\operatorname{ind}(S)=2.
\]

Let
\[
\mu:U_n\longrightarrow\mathbb Z_{\ge0}
\]
be the multiplicity function of the four-term multiset
$\{a_1,a_2,a_3,a_4\}$:
\[
\mu(x)
:=
\bigl|\{j:a_j=x\}\bigr|.
\]
Define its odd part by
\[
\nu(x):=\mu(x)-\mu(-x).
\]
Then
\[
\operatorname{supp}\nu
\subseteq
\{\pm a_1,\pm a_2,\pm a_3,\pm a_4\},
\]
so
\begin{equation}\label{eq:nu-support}
|\operatorname{supp}\nu|\le8.
\end{equation}

For $\chi\in\widehat{U_n}$,
\begin{align}
\widehat\nu(\chi)
&=
\sum_{x\in U_n}
\bigl(\mu(x)-\mu(-x)\bigr)\overline{\chi(x)}
\notag\\
&=
\bigl(1-\chi(-1)\bigr)
\sum_{j=1}^{4}\overline{\chi(a_j)}.
\label{eq:nu-fourier}
\end{align}

If $\chi$ is even, then
\[
\widehat\nu(\chi)=0
\]
by \eqref{eq:nu-fourier}.

Now let $\chi$ be odd and suppose that $\chi\notin Z_n$. Then
\[
M_n(\chi)\ne0.
\]
Since every odd character is nonprincipal,
Lemma~\ref{lem:moment-identity} gives
\[
\sum_{j=1}^{4}\overline{\chi(a_j)}=0.
\]
Hence \eqref{eq:nu-fourier} again yields
\[
\widehat\nu(\chi)=0.
\]
Therefore
\begin{equation}\label{eq:nu-fourier-support}
\operatorname{supp}\widehat\nu
\subseteq Z_n.
\end{equation}

By Lemma~\ref{lem:exceptional-spectrum},
\[
|\operatorname{supp}\widehat\nu|
\le
|Z_n|
<
\frac{\varphi(n)}{9}.
\]
If $\nu\ne0$, then the uncertainty principle
\eqref{eq:uncertainty}, together with
\eqref{eq:nu-support}, gives
\[
\varphi(n)
=
|U_n|
\le
|\operatorname{supp}\nu|\,
|\operatorname{supp}\widehat\nu|
<
8\cdot\frac{\varphi(n)}{9}
<
\varphi(n),
\]
a contradiction. Hence
\[
\nu=0.
\]
Therefore
\[
\mu(x)=\mu(-x)
\qquad (x\in U_n).
\]

Since $(n,6)=1$, the integer $n$ is odd. Hence the involution
\[
x\longmapsto -x
\]
has no fixed point on $U_n$. The equality
\[
\mu(x)=\mu(-x)
\]
therefore implies that the four-term multiset
$\{a_1,a_2,a_3,a_4\}$ is a union, with multiplicity, of opposite
pairs. In particular, $S$ contains a two-term zero-sum subsequence
\[
(x)(-x),
\]
contradicting the minimality of $S$.

Both cases lead to contradictions. Hence no counterexample exists,
and every length-four minimal zero-sum sequence over $C_n$ with
$(n,6)=1$ has index one.
\end{proof}

\begin{remark}
The proof separates the length-four index conjecture into two complementary cases. After the standard gcd reduction, the case in which at least one term is a nonunit modulo $n$ follows from the theorem of Shen--Xia--Li~\cite[Theorem~1.3]{SXL14}. The contribution of the present paper is the complete resolution of the remaining unit case by multiplicative Fourier analysis.
\end{remark}



\begin{thebibliography}{99}
\setlength{\itemsep}{0pt}

\bibitem{Apostol1976}
T.~M. Apostol,
\newblock \emph{Introduction to Analytic Number Theory},
\newblock Undergraduate Texts in Mathematics, Springer-Verlag, New York--Heidelberg, 1976.

\bibitem{CFS99}
S.~T. Chapman, M.~Freeze and W.~W. Smith,
\newblock Minimal zero-sequences and the strong Davenport constant,
\newblock \emph{Discrete Math.} \textbf{203} (1999), no.~1--3, 271--277.

\bibitem{Gao00}
W.~D. Gao,
\newblock Zero sums in finite cyclic groups,
\newblock \emph{Integers} \textbf{0} (2000), A14.

\bibitem{GaGe09}
W.~Gao and A.~Geroldinger,
\newblock On products of $k$ atoms,
\newblock \emph{Monatsh. Math.} \textbf{156} (2009), no.~2, 141--157.

\bibitem{Ge16}
F.~Ge,
\newblock Note on the index conjecture in zero-sum theory and its connection to a Dedekind-type sum,
\newblock \emph{J. Number Theory} \textbf{168} (2016), 128--134.

\bibitem{Ge18}
F.~Ge,
\newblock On the index conjecture in zero-sum theory: singular case,
\newblock \emph{Int. J. Number Theory} \textbf{14} (2018), no.~2, 355--364.

\bibitem{Ge21}
F.~Ge,
\newblock Solution to the index conjecture in zero-sum theory,
\newblock \emph{J. Combin. Theory Ser. A} \textbf{180} (2021), 105410.

\bibitem{Gero87}
A.~Geroldinger,
\newblock On non-unique factorization into irreducible elements. II,
\newblock in: \emph{Number Theory, Vol. II: Diophantine and Algebraic}, Colloq. Math. Soc. J\'anos Bolyai, vol.~51, 1990, 723--757.

\bibitem{Gero09}
A.~Geroldinger,
\newblock Additive group theory and non-unique factorizations,
\newblock in: \emph{Combinatorial Number Theory and Additive Group Theory}, Birkh\"auser, Basel, 2009, 1--86.

\bibitem{Gry12}
D.~J. Grynkiewicz,
\newblock \emph{Structural Additive Theory},
\newblock Developments in Mathematics, vol.~30, Springer, Cham, 2013.

\bibitem{GV20}
D.~J. Grynkiewicz and U.~Vishne,
\newblock The index of small length sequences,
\newblock \emph{Int. J. Algebra Comput.} \textbf{30} (2020), no.~5, 977--1014.

\bibitem{JRW13}
S.~Jabuka, S.~Robins and X.~Wang,
\newblock Heegaard Floer correction terms and Dedekind--Rademacher sums,
\newblock \emph{Int. Math. Res. Not. IMRN} \textbf{2013} (2013), no.~1, 170--183.

\bibitem{LK89}
P.~Lemke and D.~Kleitman,
\newblock An addition theorem on the integers modulo $n$,
\newblock \emph{J. Number Theory} \textbf{31} (1989), no.~3, 335--345.

\bibitem{LP13}
Y.~Li and J.~Peng,
\newblock Minimal zero-sum sequences of length four over finite cyclic groups. II,
\newblock \emph{Int. J. Number Theory} \textbf{9} (2013), no.~4, 845--866.

\bibitem{LPYZ10}
Y.~Li, C.~Plyley, P.~Yuan and X.~Zeng,
\newblock Minimal zero sum sequences of length four over finite cyclic groups,
\newblock \emph{J. Number Theory} \textbf{130} (2010), no.~9, 2033--2048.

\bibitem{Na20}
M.~B. Nathanson,
\newblock \emph{Elementary Methods in Number Theory},
\newblock Graduate Texts in Mathematics, vol.~195, Springer, New York, 2000.

\bibitem{Po04}
V.~Ponomarenko,
\newblock Minimal zero sequences of finite cyclic groups,
\newblock \emph{Integers} \textbf{4} (2004), A24, 6 pp.

\bibitem{Pe25}
A.~Pendleton,
\newblock Improved bounds for the index conjecture in zero-sum theory,
\newblock \emph{J. Number Theory} \textbf{268} (2025), 124--141.

\bibitem{SC07}
S.~Savchev and F.~Chen,
\newblock Long zero-free sequences in finite cyclic groups,
\newblock \emph{Discrete Math.} \textbf{307} (2007), no.~22, 2671--2679.

\bibitem{SX14}
C.~Shen and L.~Xia,
\newblock On the index-conjecture of length four minimal zero-sum sequences. II,
\newblock \emph{Int. J. Number Theory} \textbf{10} (2014), no.~3, 601--622.

\bibitem{SXL14}
C.~Shen, L.~Xia and Y.~Li,
\newblock On the index of length four minimal zero-sum sequences,
\newblock \emph{Colloq. Math.} \textbf{135} (2014), no.~2, 201--209.

\bibitem{Xia13}
L.~Xia,
\newblock On the index-conjecture on length four minimal zero-sum sequences,
\newblock \emph{Int. J. Number Theory} \textbf{9} (2013), no.~6, 1505--1528.

\bibitem{XS13}
L.-M. Xia and C.~Shen,
\newblock Minimal zero-sum sequences of length four over cyclic group with order $n=p^\alpha q^\beta$,
\newblock \emph{J. Number Theory} \textbf{133} (2013), no.~12, 4047--4068.

\bibitem{Yuan07}
P.~Yuan,
\newblock On the index of minimal zero-sum sequences over finite cyclic groups,
\newblock \emph{J. Combin. Theory Ser. A} \textbf{114} (2007), no.~8, 1545--1551.

\bibitem{ZQ17}
X.~Zeng and X.~Qi,
\newblock On minimal zero-sum sequences of length four over cyclic groups,
\newblock \emph{Colloq. Math.} \textbf{146} (2017), no.~2, 157--163.

\end{thebibliography}
\end{document}